\documentclass[runningheads]{llncs}

\usepackage[T1]{fontenc}
\usepackage[colorlinks=true,bookmarks=false,linkcolor=blue,urlcolor=blue,citecolor=blue,breaklinks=true]{hyperref}
\usepackage{amsmath}
\usepackage{amssymb}
\usepackage{amsthm}
\usepackage{mathrsfs}
\newcommand{\minimize}{\textrm{minimize}}
\newcommand{\R}{\mathbb{R}}

\newcommand{\In}{\mathrm{I}_n}
\newcommand{\Ip}{\mathrm{I}_p}

\newcommand{\calN}{\mathcal{N}}
\newcommand{\D}{\mathrm{D}}
\newcommand{\p}{_{\perp}}

\newcommand{\transpose}{^\top}
\newcommand{\T}{^\top}

\newcommand{\St}{\mathrm{St}}
\newcommand{\StX}{\mathrm{St}_{X\T X}}
\newcommand{\Stnp}{\mathrm{St}(n,p)}
\newcommand{\grad}{\mathrm{grad}}
\newcommand{\Proj}{\mathrm{Proj}}

\newcommand{\trace}{\mathrm{trace}}

\newcommand{\rmPhi}{\mathrm{\Phi}}
\newcommand{\Nrm}{\mathrm{N}}
\newcommand{\Trm}{\mathrm{T}}

\newcommand{\rmT}{\mathrm{T}}

\newcommand{\Sym}{\mathrm{Sym}}
\newcommand{\sym}{\operatorname{sym}}
\newcommand{\Skew}{\mathrm{Skew}}
\renewcommand{\skew}{\operatorname{skew}}

\newcommand{\Rpp}{\mathbb{R}^{p\times p}}
\newcommand{\Rnnp}{\mathbb{R}^{n\times (n-p)}}
\newcommand{\Rnpp}{\mathbb{R}^{(n-p)\times p}}
\newcommand{\rank}{\operatorname{rank}}
\newcommand{\spann}{\operatorname{span}}
\newcommand{\inner}[2]{\left\langle{#1},{#2}\right\rangle}

\newcommand{\I}{\mathrm{I}}
\newcommand{\Rnp}{{\mathbb{R}^{n \times p}}}
\newcommand{\Rnpstar}{{\mathbb{R}_{*}^{n \times p}}}

\newcommand{\Tr}{\operatorname{tr}}

\usepackage{algorithm}
\usepackage{algpseudocode}
\usepackage{todonotes}

\newcommand{\inv}{^{-1}}
\newcommand{\onehalf}{\frac{1}{2}}

\newcommand{\norm}[1]{\left\|#1\right\|}

\newcommand{\gE}{g^{\mathrm{E}}}
\newcommand{\gC}{g^{\mathrm{c}}}
\newcommand{\gbeta}{g^{\mathrm{\beta}}}

\newcommand{\ProjbetaX}{\mathrm{Proj}_{X,\beta}}
\newcommand{\normalbeta}{\mathrm{N}^\beta}

\newcommand{\nablaE}{\nabla_\mathrm{E}}
\newcommand{\nablabeta}{\nabla_{\mathrm{\beta}}}
\newcommand{\gradC}{\grad_{\mathrm{c}}}

\newcommand{\gradbeta}{\grad_{\mathrm{\beta}}}

\begin{document}
\title{Geometric design of the tangent term in landing algorithms for orthogonality constraints}
\titlerunning{Landing method using $\beta$-metric}
%
\author{Florentin Goyens\inst{1} \and P.-A. Absil\inst{1} \and
Florian Feppon\inst{2}}
\authorrunning{Goyens et al.}
%
\institute{ICTEAM Institute, UCLouvain, Louvain-la-Neuve, Belgium\\
 \and
NUMA Unit, KULeuven, Leuven, Belgium}
\maketitle              
\begin{abstract}
We propose a family a metrics over the set of full-rank $n\times p$ real matrices, and apply them to the landing framework for optimization under orthogonality constraints. The family of metrics we propose is a natural extension of the $\beta$-metric, defined on the Stiefel manifold.
\end{abstract}

We consider the constrained optimization problem     
\begin{equation}\label{eq:P}\tag{P}
\begin{aligned}
& \underset{X\in \R^{n\times p}}{\minimize}
& & f(X)
& \text{subject to}
&&& X\T X = \Ip,
\end{aligned}
\end{equation}
    where $p\leq n$ and $f\colon \Rnp \to \R$ is a smooth and possibly nonconvex function. The feasible set, called the Stiefel manifold, consists of rectangular matrices with orthonormal columns
\[
\St(n,p):=\{ X\in\R^{n\times p }\,|\, X\T X= \I_p \}.
\]   
Letting
\begin{align}\label{eq:h}
h\colon \Rnp \to \Sym(p)\colon h(X) = X\transpose X - \I_p,
\end{align}
where $\Sym(p)$ is the set of $p\times p$ symmetric matrices, the feasible set is $h\inv(0)$ and the infeasibility is measured by
\begin{align}\label{eq:N}
	\calN(X) = \onehalf \norm{h(X)}^2.
\end{align}

Algorithms for~\eqref{eq:P} fall into two broad categories: feasible and infeasible methods. Infeasible methods---typically penalty methods---appeared around the 1970's~\cite{bertsekas2014constrained,polyak1971convergence}. They consist in solving a succession of unconstrained problems, whose solutions converge towards the solution of the constrained problem under some conditions. Classical penalty methods include quadratic penalty methods and augmented Lagrangian methods. Nonsmooth penalty terms are also possible. Around the year 2000, Riemannian optimization methods blossomed~\cite{absil2008}. This new framework introduced feasible algorithms for optimization problems on smooth Riemannian manifolds. Riemannian methods generalize unconstrained optimization methods---such as gradient descent---and use differential geometry tools to keep all the iterates on the feasible set. 

Recent work~\cite{ablin2024infeasible,pmlr-v235-vary24a} recently introduced an infeasible method that exploits the local geometry of the constraint function $h$. Each step of this scheme, called a landing scheme, is the sum of two terms: a normal term that decreases infeasibility, and a tangent term that decreases the cost function $f$. A natural choice for the tangent term is the projected gradient of $f$ onto the tangent space of the current level set of $h$. Incidentally, this new framework 
turns out to be very close to the 
concept of ``null space gradient flow'' in the shape and topology optimization community developed in \cite{feppon:hal-01972915}, which 
itself borrows from diffential equation approaches to equality constrained nonlinear programming \cite{tanabe_geometric_1980,yamashita_differential_1980}.

In this work, we propose a family of (Riemannian) metrics for the ambient space $\Rnpstar$ (the set of full-rank matrices of size $n\times p$), and show that this leads to a family of tangent terms in the landing framework. In Section~\ref{sec:landing}, we review and extend the landing framework for~\eqref{eq:P}, introducing a dependence of the tangent term on the choice of a metric in the ambient space. In Section~\ref{sec:beta_stiefel}, we review the family of $\beta$-metrics on the Stiefel manifold. In Section~\ref{sec:beta_metric}, we present our main contribution: computing the tangent term of the landing in the $\beta$-metric. 

\paragraph{Notations:}
$\mathrm{D}h(Y)$ denotes the derivative of $h$ at $Y$ and $\mathrm{Sym}(p)$, resp.\ $\mathrm{Skew}(p)$, stands for the set of $p\times p$ symmetric, resp.\ skew-symmetric, matrices. Additionally, $\sym(A) = (A+A\T)/2$, $\skew(A) = (A-A\T)/2$, and $\inner{A}{B} = \trace\left(A\T B\right)$.
\section{A metric-based landing framework}
\label{sec:landing}

\subsection{Layered manifolds and metric}
In this section, we review and extend the landing framework for~\eqref{eq:P}, first introduced in~\cite{ablin2022fast} and extended to the Stiefel manifold in~\cite{ablin2024infeasible}. The landing method is based on the central observation that for $X\in \Rnpstar$, the set
\[
\St_{X\T X}:=\{ Y\in\R^{n\times p }\,|\, Y\T Y=X\T X \}
\]  
is a smooth manifold which we term a \emph{layer} manifold~\cite[Prop. 1.2]{goyens2024computing}. It is a level set of the function $h$, i.e., $\St_{X\T X}= \{Y\in \Rnp|h(Y)=h(X)\}$. For $X\in \Rnpstar$, we have $\rank(\D h(Y)) = \dim(\Sym(p))$ for all $Y\in \St_{X\T X}$, and therefore $\St_{X\T X}$ is a smooth manifold. For $X\in \Rnpstar$, the matrix $X\T X$ is symmetric and positive definite. This makes it possible to consider the linear map
\begin{align}
	\rmPhi_{X\T X}\colon \Rnp \to \Rnp \colon Y\mapsto Y(X\T X)^{1/2}.
\end{align}
The map $\rmPhi_{X\T X}$ is a diffeomorphism from $\Stnp$ to $\StX$.

 It is useful to characterize the tangent space to $\StX$ at $X$. This set does not depend on the metric.
\begin{proposition}
    The tangent space of $\StX$ at $X$ is the set    
        \begin{align} \label{eq:tgtspace1}
        \Trm_X \StX  & = \{ \xi\in\R^{n\times p}\,|\, \xi\T X+X\T \xi = 0  \} \\    
        & = \{ X(X\T X)^{-1}\Omega + \Delta\,|\, \Omega\in \Skew(p),\, \Delta
        \in \R^{n\times p} \text{ with } \Delta\T X=0\}  \label{eq:tgtspace2}\\  
    &= \{W X\,|\, W\in\Skew(n)\}, \label{eq:3form_tangent}
        \end{align}
        with dimension $np - p(p+1)/2$.
\end{proposition}
\begin{proof}
    For~\eqref{eq:tgtspace1}-\eqref{eq:tgtspace2}, see, e.g., \cite{gao_optimization_2022}. Formula~\eqref{eq:3form_tangent} appears in~\cite{gao_optimization_2022} with a minor error in the proof which we fix here (the matrix $R$ below might not be the identity). First, note that $WX\in \Trm_X \StX$ for all $W\in \Skew(n)$. Indeed, $X\T W\T X + X\T WX = 0$. To conclude, we show that 
    \[\dim \{W X\,|\, W\in\Skew(n)\} = \dim \Trm_X \StX = np - p(p+1)/2.\] 
    Let $X\in \Rnpstar$, there exists an orthogonal matrix $P\in \R^{n\times n}$ such that $P\T X = \begin{bmatrix}
 R & 0	
 \end{bmatrix}\T$ where $R\in \R^{p\times p}$ is invertible. Let \[B := P\T W P = \begin{bmatrix}
 	B_{11} & B_{12}\\
 	B_{21} & B_{22}
 \end{bmatrix}.\]
  It holds that 
 \begin{align*}
 	\dim \{W X\,|\, W\in \Skew(n)\} &= \dim \{P\T W PP\T X\,|\, W\in \Skew(n)\}\\
 	&= \dim \left\{ B \begin{bmatrix}
 R \\
  0	
 \end{bmatrix}  \,|\, B\in \Skew(n)\right\} \\
 &= \dim \left\{  \begin{bmatrix}
 B_{11}R \\
  B_{21}R	
 \end{bmatrix}  \,|\, B_{11}\in \Skew(p),~B_{21} \in \R^{(n-p)\times p}\right\} \\ \nonumber
 &= \frac{1}{2}p(p-1) + (n-p)p = np - \dfrac{1}{2}p(p+1). \qedhere
 \end{align*}
\end{proof}

We use the notation
\begin{align}\label{eq:frobenius_metric}
\inner{\cdot}{\cdot} \colon \Rnp\times \Rnp \to \R \colon	\inner{X}{Y} &= 	\trace	\left(X\transpose Y\right)
\end{align}
to denote the Frobenius inner product. However, our aim is to generalize the landing algorithm and show that one can define a landing algorithm for any metric in the ambient space. Let $g$ denote a Riemannian metric on $\Rnpstar$. The metric $g$ allows to define several objects useful for optimization. First, the metric defines a \emph{normal space} to the manifold $\St_{X\T X}$ at $X$,
	\begin{align}
		\Nrm^g_X \St_{X\transpose X}:= \left\lbrace U\in \Rnp \colon g(U,\xi) = 0 \text{ for all } \xi\in \Trm_X \StX\right \rbrace.
	\end{align}
	The \emph{unconstrained Riemannian gradient} of $f$, written $\nabla_g f$, also depends on the metric $g$. Given $X\in \Rnpstar$, $\nabla_g f(X)$ is the unique element of $\rmT_X \Rnpstar \simeq \Rnp$ that satisfies, for all $\xi \in \rmT_X \Rnpstar$,
	\begin{align}\label{eq:unconstrained_gradient}
		\mathrm{D}f(X)[\xi] = g(\xi, \nabla_g f(X)).
	\end{align}
	The Riemannian metric induces a \emph{constrained Riemannian gradient}, written $\mathrm{grad}_g f$. Given $X\in \Rnpstar$, $\mathrm{grad}_g f(X)$ is the unique vector in $\rmT_X\StX$ that satisfies, for all $\xi\in \Trm_X \StX$, 
	\begin{align}\label{eq:riemannian_gradient}
		\D f(X)[\xi] = g(\xi, \grad_g f(X)).
		\end{align}
  The orthogonal projection onto $\Trm_X \StX$ with respect to the metric $g$, written $\Proj_{X,g} \colon \Rnp \to \Trm_X \StX$, is characterized by
		\begin{align}\label{eq:projection}
			g(\xi, Z-\Proj_{X,g}(Z))=0,
		\end{align}
		for all $\xi \in \Trm_X \StX$ and all $Z\in \Rnp$ (see, e.g.,~\cite[(3.37)]{absil2008}). The constrained and unconstrained Riemannian gradients of $f$ are related by 
		$$
\grad_g f(X) = \Proj_{X,g}(\nabla_g f(X)).$$
 
 \subsection{The landing algorithm}
 The proposed landing algorithm for a metric $g$ is presented in Algorithm~\ref{algo:landing_framework}. It makes a step along the weighted sum of two terms. 
 The first term is the constrained Riemannian gradient step, $-\grad_g f(X_k)$, which belongs to the tangent space to $\StX$ at $X_k$. It decreases the cost function in a direction that
 preserves the constraint function $h$ at the first order, in the sense that $$\D h(X)[\grad_g f(X)]=0.$$ The second term decreases infeasibility. As in~\cite{ablin2024infeasible,pmlr-v235-vary24a}, we set this term to be the negative Euclidean gradient of the infeasibility function $\mathcal{N}$ (eq.\,\eqref{eq:N}), given by
 \begin{align}
    \label{eqn:3lio6}
-	 \nabla \calN(X) &=- X(X\T X - \I_p),
\end{align}  
that is the gradient defined with respect to the standard Frobenius metric. Note that the metric chosen to compute the gradient of $\calN$ can be 
selected independently of $g$, since the only important property is that this term leads to a descent of $\mathcal{N}$.

\begin{algorithm}
\caption{Landing framework}\label{algo:landing_framework}
\begin{algorithmic}[1]
\State \textbf{Given:} Metric $g$ on $\Rnpstar$, tolerance $\varepsilon>0$, $X_0\in \Rnpstar$, parameter $\omega>0$.
\While{$ \norm{\grad_g f(X_k)} + \norm{h(X_k)} >\varepsilon$}
\State Set stepsize $\eta_k>0$
\State $X_{k+1} = X_k - \eta_k \left( \grad_g f(X_k) + \omega \nabla \calN(X_k)\right)$
\State $k \leftarrow k +1$
\EndWhile
\end{algorithmic}
\end{algorithm}

\section{Metrics on the feasible set $\Stnp$}
\label{sec:beta_stiefel}
The proposed landing algorithm requires to choose a metric $g$ on $\Rnpstar$	, which defines the tangent term
$\grad_g f(X)$. Before we propose a metric for the ambient space $\Rnpstar$, we
discuss possible choices of metric on the feasible set $\St(n,p)$.

To turn the set $\Stnp$ into a Riemannian manifold, we have to endow each tangent space with an inner product. The two mainstream choices are called the Euclidean and canonical metrics on $\Stnp$; they are subsumed by a family of metrics called the $\beta$-metrics~\cite{mataigne2024efficient}, originally introduced in~\cite{huper2021lagrangian} with a different notation. The $\beta$-metric is defined for $X\in \Stnp$ and $\beta>0$ as
\begin{align}\label{eq:betametric_stiefel}
	\gbeta_X\colon \Trm_X \Stnp \times \Trm_X \Stnp \to \R \colon (\xi,\zeta) \mapsto \inner{\xi}{\left(\I_n-(1-\beta)X X\transpose\right)\zeta }.
\end{align}
Given $X\p \in \Rnnp$, a basis for the orthogonal complement of $\spann(X)$, satisfying 
\[
	X\p \T X\p = \I_{n-p} \qquad \textrm{ and }\qquad  X\transpose X\p = 0,
\]
we can uniquely decompose any $\xi,\zeta \in \Trm_X \Stnp$ into
\begin{align}
	\xi &= X \Omega_\xi + X\p B_\xi\\
	\zeta &= X \Omega_\zeta + X\p B_\zeta,
\end{align} 
with $\Omega_\xi, \Omega_\zeta \in \Skew(p)$ and $B_\xi,B_\zeta \in \R^{(n-p)\times p}$. Using this notation, the $\beta$-metric is given by
\begin{align}
	\gbeta_X(\xi,\zeta) &= \beta\inner{\Omega_\xi}{\Omega_\zeta} +  \inner{B_\xi}{B_\zeta},~~~ \textrm{ for } X\in \Stnp.
\end{align}
The Euclidean metric corresponds to $\beta=1$, which we write \[\gE(\xi,\zeta) =
\inner{\xi}{\zeta}.\]
 For $\beta=\frac{1}{2}$, we obtain the canonical metric
\cite{edelman1998geometry}, written \[\gC_X(\xi,\zeta) = \inner{\xi}{\left(\In-\frac{1}{2}X X\transpose\right)\zeta }.\]

\section{Extending the $\beta$-metric to the ambient space}
\label{sec:beta_metric}

In this section, we extend the $\beta$-metric~\eqref{eq:betametric_stiefel} from the feasible set $\Stnp$ to the ambient space $\Rnpstar$. We follow the approach proposed in~\cite{gao_optimization_2022} for the extension of the canonical metric. First, extend~\eqref{eq:betametric_stiefel} to $\xi, \zeta \in \mathbb{R}^{n\times p}$, and observe that this remains a well-defined inner product. Next, for $X\in\mathbb{R}_*^{n\times p}$, consider the pullback of the extended~\eqref{eq:betametric_stiefel} through $\rmPhi_{X^\top X}^{-1}$: 
\begin{align}\label{eq:beta_metric}
	\gbeta_X(\xi, \zeta)&:= \gbeta_{\rmPhi\inv_{X\T X}(X)} \left(\D \rmPhi\inv_{X\T X}(X)[\xi],\D\rmPhi\inv_{X\T X}(X)[\zeta]\right)\\
	 &= \inner{\xi}{(\In - (1-\beta)X (X\T X)\inv X\T) \zeta (X\T X)\inv},
\end{align}
for all $\xi, \zeta \in \rmT_X \Rnpstar \simeq \Rnpstar$. This defines a Riemannian metric on $\mathbb{R}_*^{n\times p}$.

 This metric depends on the reference point $X$. Therefore, there is no value of $\beta$ for which this is equivalent to the Frobenius metric $\gE(\xi,\zeta)= \trace(\xi\T \zeta)$. For $\beta = \frac{1}{2}$, we recover the extension of the canonical metric proposed in~\cite{gao_optimization_2022} for $X\in\Rnpstar$: 
 \begin{equation}
\gC_X(\xi,\zeta):= \langle \xi,(\In-\frac{1}{2}X(X\T X)\inv X\T)\zeta (X\T X)\inv
\rangle, \qquad \xi,\zeta\in\R^{n\times p}.
\end{equation}
The following proposition is a convenient formula for the $\beta$-metric.
\begin{proposition}
Let $X\in \Rnpstar$ and $X\p \in \R^{n\times (n-p)}$ be such that
 $\begin{bmatrix}
		X & X\p
\end{bmatrix}$
 is invertible and $X\T X\p = 0$. Let $\eta,\xi \in \Rnp$, which are decomposed as
\begin{align}
	\eta &= XA_\eta + X\p B_\eta & \text{ and } && \xi = XA_\xi + X\p B_\xi,
\end{align}
with $A_\eta,A_\xi \in \Rpp$ and $B_\eta,B_\xi \in \R^{(n-p)\times p}$. We have
\begin{align}
	g^\beta_X(\eta,\xi) = \beta \Tr \left( A_\eta\T X\T X A_\xi (X\T X)^{-1}\right) + \Tr \left( B_\eta\T X\T\p X\p B_\xi (X\T X)^{-1}\right). 
\end{align}
\end{proposition}
\subsection{The tangent direction}
Our aim is now to compute the term $\gradbeta f(X) = \ProjbetaX(\nablabeta f(X))$, in order to define the landing algorithm with respect to the $\beta$-metric. Remarkably, the normal space corresponding to the metric $\gbeta$ does not depend on the parameter $\beta$.
\begin{proposition}
For $\beta>0$ and $X\in \Rnpstar$, the normal space to $\St_{X\T X}$ with respect to the metric $g^\beta$ is given by
\[
\normalbeta_X\St_{X\T X}=\{ X(X\T X)^{-1}S\,|\, S\in\Sym(p) \}.
\] 
\end{proposition}
\begin{proof}
	Let $\eta = XA + X\p B \in \normalbeta_X \StX$. By definition, we have $\gbeta(\eta,\xi) = 0$ for all $\xi \in \Trm_X \StX$, that is, all vectors of the form  \[\xi= X(X\T X)\inv \Omega_\xi + X\p B_\xi \] for some $\Omega_\xi \in \Skew(p)$, $B_\xi \in \Rnpp$. We find
	\begin{align}
		\gbeta(\eta,\xi) &= \beta \Tr \left[ A\T X\T X (X\T X)\inv \Omega_\xi (X\T X)\inv\right] + \Tr \left[ B\T X\p\T X\p B_\xi (X\T X)\inv \right]\\
		& = \beta \Tr \left[ A\T \Omega_\xi (X\T X)\inv\right] + \Tr \left[ B\T X\p\T X\p B_\xi (X\T X)\inv \right]=0
	\end{align}
	for all $\Omega_\xi \in \Skew(p)$ and $B_\xi \in \Rnpp$. Let $B_\xi=0$, we have 
	\begin{align}
		0 &= \Tr \left[ A\T \Omega_\xi (X\T X)\inv\right]\\
		&= \inner{A (X\T X)\inv}{\Omega_\xi}
	\end{align}
	for all $\Omega_\xi \in \Skew(p)$. This holds if and only if $A (X\T X)\inv \in \Sym(p)$, which is equivalent to $A =(X\T X)\inv S$ where $S\in \Sym(p)$. Using $\Omega_\xi =0$, we find that $B=0$. Thus, the normal vectors are the vectors of the form $X (X\T X)\inv S$ with $S \in \Sym(p)$.
\end{proof}

\begin{proposition}
    The orthogonal projection onto the tangent space $T_X\StX$ with respect to
    the metric $\gbeta$ is given for $Z\in\R^{n\times p}$ by    
    \begin{equation}\label{eq:proj_beta}
    \ProjbetaX\left(Z\right)=X(X\T X)^{-1}\skew(X\T Z) + (\In-X(X\T X)^{-1}X\T )Z.
    \end{equation}
\end{proposition}
\begin{proof}
	Let $Z = XA + X\p B \in \Rnp$. We want to find $\eta \in \normalbeta_X\StX$ such that $Z-\eta \in \Trm_X \StX$, which implies $\ProjbetaX(Z) = Z-\eta$ in view of~\eqref{eq:projection}. Let 
	\[\ProjbetaX(Z)= X (X\T X)\inv \Omega_\xi + X\p B_\xi\] 
	for some $\Omega_\xi \in \Skew(p)$, $B_\xi \in \Rnpp$
	and $$\eta = X (X\T X)\inv S$$ with $S\in \Sym(p)$. We impose
	\begin{align}
		XA + X\p B - X (X\T X)\inv S = X (X\T X)\inv \Omega_\xi + X\p B_\xi.
	\end{align}
	This is satisfied for $B = B_\xi$ and $S + \Omega_\xi = X\T X A$, which implies $S= \sym(X\T XA)$ and $\Omega_\xi = \skew(X\T XA)$. In conclusion, 
	\begin{align}
		\ProjbetaX(Z) = X(X\T X)^{-1}\skew(X\T XA) + X\p B.
	\end{align}
	The result follows from $X\T Z = X\T XA$ and $X\p B = (\In-X(X\T X)^{-1}X\T)Z$.
	\end{proof}
\begin{proposition}
	The unconstrained Riemannian gradient in the metric $\gbeta$ is given by
	\begin{align}\label{eq:nabla_beta}
		\nablabeta f(X) &= \left( \In + \frac{(1-\beta)}{\beta} X (X\T X)\inv X\T \right) \nablaE f(X) X\T X,
	\end{align}
	where $\nablaE f$ is the unconstrained Riemannian gradient in the Euclidean metric. 
	\end{proposition}
	\begin{proof}
		By definition, for all $X\in \Rnpstar$ and $Z\in \Rnp$,
		\begin{align}
			\inner{\nablaE f(X)}{Z} &= \D f(X)[Z]  \\
			&= \gbeta(\nablabeta f(X),Z) \\
			 &= \inner{\left(\In-(1-\beta)X(X\transpose X)\inv X\transpose\right)\nablabeta f(X) (X\transpose X)\inv}{Z}.
		\end{align}
By identification, this gives,
\begin{equation}
	\nablaE f(X) = \left(\In-(1-\beta)X(X\transpose X)\inv X\transpose\right)\nablabeta f(X) (X\transpose X)\inv.
	\end{equation}
To obtain~\eqref{eq:nabla_beta}, it is readily verified that 
\[\left(\In-(1-\beta)X(X\transpose X)\inv X\transpose\right) \left(\In +\frac{(1-\beta)}{\beta}X(X\transpose X)\inv X\transpose\right) = \In. \qedhere\]
	\end{proof}
We conclude with the expression for the constrained Riemannian gradient on the layer manifold $ \StX$ for the metric $\gbeta$. The case $\beta = \frac12$ recovers~\cite[Prop.~4]{gao_optimization_2022}.
\begin{proposition}
For $\beta>0$, the constrained Riemannian gradient at $X\in \Rnpstar$ with respect to the metric $\gbeta$ is given by
\begin{align}
	\grad_{\beta} f(X) &= \ProjbetaX (\nablabeta f(X)) \label{eq:grad_beta}\\
	&= \nablaE f(X) X\T X - \frac{1}{2\beta} X \nablaE f(X)\T X \nonumber \\
	&~~~+ \left(\frac{1}{2\beta} -1\right)X (X\T X)\inv X\T \nablaE f(X) X\T X.
\end{align} 
For $\beta=\frac{1}{2}$, this simplifies to 
\begin{align}
\gradC f(X) &= 	\nablaE f(X) X\T X -  X \nablaE f(X)\T X \\
&= 2 \skew\left( \nablaE f(X) X\T\right)X.
\end{align}
\end{proposition}
\begin{proof}
	To apply~\eqref{eq:grad_beta} directly, we combine~\eqref{eq:nabla_beta} and~\eqref{eq:proj_beta}. We find
	\begin{align*}
			\gradbeta f(X) &= X(X\T X)^{-1}\skew\left(X\T \nablaE f(X) X\T X\right)\\
						&~~+ \frac{(1-\beta)}{\beta} X(X\T X)^{-1}\skew\left( X\T \nablaE f(X) X\T X\right) \\
						&~~+ \left( \In - X (X\T X)\inv X\T \right) \nablaE f(X) X\T X\\
						&= \frac{1}{\beta}X(X\T X)^{-1}\skew\left(X\T \nablaE f(X) X\T X\right)\\
						&~~+ \nablaE f(X) X\T X - X (X\T X)\inv X\T \nablaE f(X) X\T X\\
						&= \frac{1}{2\beta}X(X\T X)^{-1} \left(X\T \nablaE f(X) X\T X\right)\\
						&~~ - \frac{1}{2\beta}X(X\T X)^{-1} \left(X\T \nablaE f(X) X\T X\right)\T \\
						&~~ + \nablaE f(X) X\T X - X (X\T X)\inv X\T \nablaE f(X) X\T X\\
						&= \frac{1}{2\beta}X(X\T X)^{-1} X\T \nablaE f(X) X\T X - \frac{1}{2\beta}X  \nablaE f(X)\T X \\ \nonumber
						&~~ + \nablaE f(X) X\T X - X (X\T X)\inv X\T \nablaE f(X) X\T X. 
\qedhere
	\end{align*}
\end{proof}

In conclusion, we have obtained a constrained Riemannian gradient $\gradbeta f(X)$ which is easily computable from $\nablaE f(X)$, the unconstrained Riemannian gradient in the Euclidean metric.

\begin{credits}
\subsubsection{\ackname} This work was supported by the Fonds de la Recherche Scientifique-FNRS under Grant no T.0001.23. F. Feppon was supported by the Flanders Research Foundation (FWO) under Grant G001824N.

\subsubsection{\discintname} The authors have no competing interests to declare that are
relevant to the content of this article.
\end{credits}
%
%
%
 \bibliographystyle{splncs04}
 \bibliography{references.bib}

\end{document}